\documentclass[12pt]{amsart}

\usepackage{accents}
\usepackage{appendix}
\usepackage{amsfonts}
\usepackage{amsmath}
\usepackage{amssymb}	
\usepackage{amsthm,bm}
\usepackage{array,booktabs,multirow}
\usepackage{braket}
\usepackage{centernot}
\usepackage{cite}
\usepackage{comment}
\usepackage{dsfont}
\usepackage{mathalfa}
\usepackage[shortlabels]{enumitem} 
\usepackage{etoolbox}
\usepackage{float}
\usepackage[hang, flushmargin]{footmisc}
\usepackage{latexsym}
\usepackage{lipsum}
\usepackage{needspace}
\usepackage{tikz}
\usepackage{hyperref}
\usetikzlibrary{matrix,arrows}  
\usepackage{kbordermatrix}

\makeatletter
\def\namedlabel#1#2{\begingroup
   \def\@currentlabel{#2}%
   \label{#1}\endgroup
}
\makeatother

\renewcommand{\kbldelim}{(}
\renewcommand{\kbrdelim}{)}

\theoremstyle{plain}
\newtheorem{thm}{Theorem}[section]

\newtheorem{lem}[thm]{Lemma}
\newtheorem{prop}[thm]{Proposition}

\theoremstyle{definition}

\newtheorem{defn}[thm]{Definition}
\newtheorem{exam}[thm]{Example}

\theoremstyle{remark}

\setlist[enumerate,1]{leftmargin=2em}

\def\H{\mathfrak H}

\def\F{\mathbb F}
\def\Z{\mathbb Z}

\def\e{\varepsilon}

\title[The universal DAHA of type $(C_1^\vee,C_1)$ and Leonard triples]{The universal DAHA of type $(C_1^\vee,C_1)$ and Leonard triples}

\author{Hau-Wen Huang}
\address{
Hau-Wen Huang\\
Department of Mathematics\\
National Central University\\
Chung-Li 32001 Taiwan
}
\email{hauwenh@math.ncu.edu.tw}

\begin{document}
\begin{abstract}
Assume that $\F$ is an algebraically closed field and $q$ is a nonzero scalar in $\F$ that is not a root of unity. 
The universal Askey--Wilson algebra $\triangle_q$ is a unital associative $\F$-algebra generated by $A,B,C$ and the relations state that each of 
$$
A+\frac{q BC-q^{-1} CB}{q^2-q^{-2}},
\qquad 
B+\frac{q CA-q^{-1} AC}{q^2-q^{-2}},
\qquad 
C+\frac{q AB-q^{-1} BA}{q^2-q^{-2}}
$$ 
is central in $\triangle_q$. 
The universal DAHA $\mathfrak H_q$ of type $(C_1^\vee,C_1)$ is a unital associative $\F$-algebra generated by $\{t_i^{\pm 1}\}_{i=0}^3$ and the relations state that 
\begin{gather*}
t_it_i^{-1}=t_i^{-1} t_i=1
\quad 
\hbox{for all $i=0,1,2,3$};
\\
\hbox{$t_i+t_i^{-1}$ is central} 
\quad 
\hbox{for all $i=0,1,2,3$};
\\
t_0t_1t_2t_3=q^{-1}.
\end{gather*}
It was given an $\F$-algebra homomorphism $\triangle_q\to \H_q$ that sends  
\begin{eqnarray*}
A &\mapsto & t_1 t_0+(t_1 t_0)^{-1},
\\
B &\mapsto & t_3 t_0+(t_3 t_0)^{-1},
\\
C &\mapsto & t_2 t_0+(t_2 t_0)^{-1}.
\end{eqnarray*}
Therefore any $\H_q$-module can be considered as a $\triangle_q$-module.
Let $V$ denote a finite-dimensional irreducible $\H_q$-module. In this paper we show that $A,B,C$ are diagonalizable on $V$ if and only if $A,B,C$ act as Leonard triples 
on all composition factors of the $\triangle_q$-module $V$. 
\end{abstract}

\maketitle



\allowdisplaybreaks

\section{Introduction}\label{s:introduction}

Throughout this paper, we adopt the following conventions. Assume that $\F$ is an algebraically closed field and fix a nonzero scalar $q\in \F$ that is not a root of unity.

We begin this paper by recalling the notion of Leonard pairs and Leonard triples. We will use the following terms. A square matrix is said to be {\it tridiagonal} if each nonzero entry lies on either the diagonal, the superdiagonal, or the subdiagonal. A tridiagonal matrix is said to be {\it irreducible} if each entry on the superdiagonal is nonzero and each entry on the subdiagonal is nonzero.

\begin{defn}
[Definition 1.1, \cite{lp2001}]
\label{defn:lp}
Let $V$ denote a vector space over $\F$ with finite positive dimension. By a {\it Leonard pair} on $V$, we mean a pair of linear operators $L: V \to V$ and $L^* : V \to V$ that satisfy both (i), (ii) below.
\begin{enumerate}
\item There exists a basis for $V$ with respect to which the matrix representing $L$ is diagonal and the matrix representing $L^*$ is irreducible tridiagonal. 

\item There exists a basis for $V$ with respect to which the matrix representing $L^*$  is diagonal and the matrix representing $L$ is irreducible tridiagonal.
\end{enumerate}
\end{defn}

\begin{defn}
[Definition 1.2, \cite{cur2007}]
\label{defn:lt}
Let $V$ denote a vector space over $\F$ with finite positive dimension. 
By a {\it Leonard triple} on $V$, we mean a triple of linear operators $L:V\to V,L^* :V\to V,L^\e:V\to V$ that satisfy the conditions (i)--(iii) below.
\begin{enumerate}
\item There exists a basis for $V$ with respect to which the matrix representing $L$ is diagonal and the matrices representing $L^* $ and $L^\e$ are irreducible tridiagonal.

\item There exists a basis for $V$ with respect to which the matrix representing $L^* $ is diagonal and the matrices representing $L^\e$ and $L$ are irreducible tridiagonal.

\item There exists a basis for $V$ with respect to which the matrix representing $L^\e$ is diagonal and the matrices representing $L$ and $L^* $ are irreducible tridiagonal.
\end{enumerate}
\end{defn}

In \cite{Leo1982}, Leonard gave a characterization of the $q$-Racah polynomials. The Leonard pairs provide a linear algebra interpretation of Leonard's theorem \cite[Appendix A]{lp2001}. The Leonard triples are a natural generalization of Leonard pairs.

Next we recall the universal Askey--Wilson algebra and its connections to the Leonard pairs and Leonard triples. 
In \cite{hidden_sym}, Zhedanov proposed the Askey--Wilson algebras to link the representation theory and the Askey--Wilson polynomials. In \cite{uaw2011}, Terwilliger introduced the universal Askey--Wilson algebra as a central extension of the Askey--Wilson algebras and the definition is as follows.

\begin{defn}
[Definition 2.1, \cite{uaw2011}]
\label{defn:UAW}
The  {\it universal Askey--Wilson algebra} $\triangle_q$ is a unital associative $\F$-algebra defined by generators and relations in the following way. The generators are $A, B, C$ and the relations state that each of
\begin{gather}\label{UAW_central}
A+\frac{q BC-q^{-1} CB}{q^2-q^{-2}},
\qquad 
B+\frac{q CA-q^{-1} AC}{q^2-q^{-2}},
\qquad 
C+\frac{q AB-q^{-1} BA}{q^2-q^{-2}}
\end{gather}
is central in $\triangle_q$. 
\end{defn}

\noindent Let $V$ denote a vector space over $\F$ with finite positive dimension. An eigenvalue $\theta$ of a linear operator $L:V\to V$ is said to be {\it multiplicity-free} if the algebraic multiplicity of $\theta$ is equal to $1$. A linear operator $L:V\to V$ is said to be {\it  multiplicity-free} if all eigenvalues of $L$ are multiplicity-free. The representation theory of $\triangle_q$ is related to Leonard pairs and Leonard triples in the following ways:

\begin{lem}
[Theorem 6.2, \cite{lp&awrelation}; Theorem 5.2, \cite{Huang:2015}]
\label{lem:lp}
If $V$ is a finite-dimensional irreducible $\triangle_q$-module then the following are equivalent:
\begin{enumerate}
\item $A,B$ {\rm (}resp. $B,C${\rm )} {\rm (}resp. $C,A${\rm )} are diagonalizable on $V$.

\item $A,B$ {\rm (}resp. $B,C${\rm )} {\rm (}resp. $C,A${\rm )} are multiplicity-free on $V$.

\item $A,B$ {\rm (}resp. $B,C${\rm )} {\rm (}resp. $C,A${\rm )} act as a Leonard pair on $V$.
\end{enumerate}
\end{lem}

\begin{lem}
[Theorem 5.3, \cite{Huang:2015}]
\label{lem:lt}
If $V$ is a finite-dimensional irreducible $\triangle_q$-module then the following are equivalent:
\begin{enumerate}
\item $A,B,C$ are diagonalizable on $V$.

\item $A,B,C$ are multiplicity-free on $V$.

\item $A,B,C$ act as a Leonard triple on $V$.
\end{enumerate}
\end{lem}

\noindent If the equivalent statements (i)--(iii) of Lemma \ref{lem:lp} hold, then the Leonard pairs are of $q$-Racah type \cite{lp2001}. If the equivalent statements (i)--(iii) of Lemma \ref{lem:lt}, then the Leonard triple is of $q$-Racah type \cite{Huang:2012}.

Next we recall the universal DAHA (double affine Hecke algebra) of type $(C_1^\vee,C_1)$ and its connection to the universal Askey--Wilson algebra $\triangle_q$. In \cite{Cherednik:1995,Cherednik:1992}  the DAHAs  were introduced by  Cherednik in connection with quantum affine Knizhni--Zamolodchikov equations and Macdonald eigenvalue problems. The DAHAs of type $(C_1^\vee,C_1)$ are closely related to the Askey--Wilson polynomials and their nonsymmetric counterpart \cite{sahi1,sahi2,DAHA&OP_book,Noumi04,koo07,koo08,daha&Z3,DAHA2013}. In\cite{DAHA2013}, the universal DAHA of type $(C_1^\vee,C_1)$ was introduced by  Terwilliger as a central extension of the DAHAs of type $(C_1^\vee,C_1)$ and the definition is as follows.

\begin{defn}
[Definition 3.1, \cite{DAHA2013}] 
\label{defn:H}
The {\it universal DAHA $\H_q$ of type $(C_1^\vee,C_1)$} is a unital associative $\F$-algebra defined by generators and relations. The generators are $\{t_i^{\pm 1}\}_{i=0}^3$ and the relations state that  
\begin{gather}
t_it_i^{-1}=t_i^{-1} t_i=1
\quad 
\hbox{for all $i=0,1,2,3$};
\nonumber
\\
\hbox{$t_i+t_i^{-1}$ is central} 
\quad 
\hbox{for all $i=0,1,2,3$};
\label{ti+tiinv}
\\
t_0t_1t_2t_3=q^{-1}.
\label{t0123}
\end{gather}
\end{defn}

\noindent The algebra $\triangle_q$ is related to the algebra $\H_q$ in the following way and this result can be considered as a universal analogue of \cite[Corollary 6.3]{koo07}.

\begin{thm}
[Theorem 4.1, \cite{DAHA2013}]
\label{thm:hom}
There exists a unique $\F$-algebra homomorphism $\triangle_q\to \H_q$ that sends 
\begin{eqnarray*}
A &\mapsto & t_1 t_0+(t_1 t_0)^{-1},
\\
B &\mapsto & t_3 t_0+(t_3 t_0)^{-1},
\\
C &\mapsto & t_2 t_0+(t_2 t_0)^{-1}.
\end{eqnarray*}
\end{thm}

By virtue of Theorem \ref{thm:hom}, each $\H_q$-module can be viewed as a $\triangle_q$-module. 
Suppose that $V$ is a finite-dimensional irreducible $\H_q$-module. In \cite{daha&LP}, it was shown that if 
\begin{enumerate}
\item[(a)\namedlabel{t0}{(a)}] $t_0$ has two distinct eigenvalues on $V$;

\item[(b)\namedlabel{XY}{(b)}] $t_0 t_1$ and $t_3 t_0$ are diagonalizable on $V$,
\end{enumerate}
then $A$ and $B$ act as Leonard pairs on the eigenspaces of $t_0$ in $V$.
In \cite{Huang:AW&DAHAmodule}, it was given a classification of the lattices of $\triangle_q$-submodules of finite-dimensional irreducible $\H_q$-modules. 
As a consequence of \cite{Huang:AW&DAHAmodule}, the $\triangle_q$-module $V$ is completely reducible if and only if $t_0$ is diagonalizable on $V$; in this case the irreducible $\triangle_q$-submodules of $V$ are the eigenspaces of $t_0$ in $V$. Note that the condition \ref{t0} is a sufficient condition for $t_0$ as diagonalizable on $V$ but not a necessary condition. See \cite[\S5]{Huang:AW&DAHAmodule} for details.
In this paper we are devoted to proving the following result:

\begin{thm}\label{thm:diagonal}
If $V$ is a finite-dimensional irreducible $\H_q$-module then the following are equivalent:
\begin{enumerate}
\item $A$ {\rm (}resp. $B${\rm )} {\rm (}resp. $C${\rm )}
is diagonalizable on $V$.

\item $A$ {\rm (}resp. $B${\rm )} {\rm (}resp. $C${\rm )} is diagonalizable on all composition factors of the $\triangle_q$-module $V$.

\item $A$ {\rm (}resp. $B${\rm )} {\rm (}resp. $C${\rm )} is multiplicity-free on all composition factors of the $\triangle_q$-module $V$.
\end{enumerate}
\end{thm}

\noindent 
Combined with Lemmas \ref{lem:lp} and \ref{lem:lt} we have the following byproducts:

\begin{thm}\label{thm:lp}
If $V$ is a finite-dimensional irreducible $\H_q$-module then the following are equivalent:
\begin{enumerate}
\item $A,B$ {\rm (}resp. $B,C${\rm )} {\rm (}resp. $C,A${\rm )}
are diagonalizable on $V$.

\item $A,B$ {\rm (}resp. $B,C${\rm )} {\rm (}resp. $C,A${\rm )} are diagonalizable on all composition factors of the $\triangle_q$-module $V$.

\item $A,B$ {\rm (}resp. $B,C${\rm )} {\rm (}resp. $C,A${\rm )} are multiplicity-free on all composition factors of the $\triangle_q$-module $V$.

\item $A,B$ {\rm (}resp. $B,C${\rm )} {\rm (}resp. $C,A${\rm )} act as Leonard pairs on all composition factors of the $\triangle_q$-module $V$.
\end{enumerate}
\end{thm}

\begin{thm}\label{thm:lt}
If $V$ is a finite-dimensional irreducible $\H_q$-module then the following are equivalent:
\begin{enumerate}
\item $A,B,C$
are diagonalizable on $V$.

\item $A,B,C$ are diagonalizable on all composition factors of the $\triangle_q$-module $V$.

\item $A,B,C$ are multiplicity-free on all composition factors of the $\triangle_q$-module $V$.

\item $A,B,C$ act as Leonard triples on all composition factors of the $\triangle_q$-module $V$.
\end{enumerate}
\end{thm}

\noindent Note that the condition \ref{XY} is a sufficient condition for $A$ and $B$ as diagonalizable on $V$ but not a necessary condition. See Examples \ref{exam:b_E} and \ref{exam:b_O}.

The paper is organized as follows. In \S\ref{s:AWmodule} we recall some sufficient and necessary conditions for $A,B,C$ as diagonalizable on finite-dimensional irreducible $\triangle_q$-modules. In \S\ref{s:even} we prove Theorem \ref{thm:diagonal} in the even-dimensional case. In \S\ref{s:odd} we prove Theorem \ref{thm:diagonal} in the odd-dimensional case.

\section{The conditions for $A,B,C$ as diagonalizable on finite-dimensional irreducible $\triangle_q$-modules}\label{s:AWmodule}

In this section, we recall the equivalent conditions for $A,B,C$ as diagonalizable on finite-dimensional irreducible $\triangle_q$-modules. We begin by some facts concerning the finite-dimensional irreducible $\triangle_q$-modules.

Define the elements $\alpha,\beta,\gamma$ of $\triangle_q$ obtained by multiplying the three elements (\ref{UAW_central}) by $q+q^{-1}$, respectively. In other words, 
\begin{eqnarray*}
\frac{\alpha}{q+q^{-1}} 
&=&
A+\frac{q BC-q^{-1} CB}{q^2-q^{-2}},
\\
\frac{\beta}{q+q^{-1}} 
&=&
B+\frac{q CA-q^{-1} AC}{q^2-q^{-2}},
\\
\frac{\gamma}{q+q^{-1}} 
&=& 
C+\frac{q AB-q^{-1} BA}{q^2-q^{-2}}.
\end{eqnarray*}
By Definition \ref{defn:UAW} the elements $\alpha,\beta,\gamma$ are central in $\triangle_q$.

\begin{prop}
[\S 4.1, \cite{Huang:2015}]
\label{prop:UAWd}
For any nonzero scalars $a,b,c\in \F$ and any integer $d\geq 0$, there exists a $(d+1)$-dimensional $\triangle_q$-module $V_d(a,b,c)$  satisfying the following conditions {\rm (i)}, {\rm (ii)}: 
\begin{enumerate}
\item There exists an $\F$-basis for $V_d(a,b,c)$ with respect to which the matrices representing $A$ and $B$ are 
$$
\begin{pmatrix}
\theta_0 & & &  &{\bf 0}
\\ 
1 &\theta_1 
\\
&1 &\theta_2
 \\
& &\ddots &\ddots
 \\
{\bf 0} & & &1 &\theta_d
\end{pmatrix},
\qquad 
\begin{pmatrix}
\theta_0^* &\varphi_1 &  & &{\bf 0}
\\ 
 &\theta_1^* &\varphi_2
\\
 &  &\theta_2^* &\ddots
 \\
 & & &\ddots &\varphi_d
 \\
{\bf 0}  & & & &\theta_d^*
\end{pmatrix},
$$
respectively, where 
\begin{align*}
\theta_i 
&=
a q^{2i-d}+a^{-1} q^{d-2i}
\qquad 
\hbox{for $i=0,1,\ldots,d$},
\\
\theta^*_i 
&=
b q^{2i-d}+b^{-1} q^{d-2i}
\qquad 
\hbox{for $i=0,1,\ldots,d$},
\\
\varphi_i 
&=
a^{-1} b^{-1} q^{d+1}
(q^i-q^{-i})
(q^{i-d-1}-q^{d-i+1})
\\
&\qquad \times \;
(q^{-i}-a b c q^{i-d-1})
(q^{-i}-a b c^{-1} q^{i-d-1})
\qquad 
\hbox{for $i=1,2,\ldots,d$}.
\end{align*}

\item The elements $\alpha,\beta,\gamma$ act on $V_d (a,b,c)$ as scalar multiplication by
\begin{align*}
(b+b^{-1}) (c+c^{-1})+(a+a^{-1}) (q^{d+1}+q^{-d-1}),
\\
(c+c^{-1}) (a+a^{-1})+(b+b^{-1}) (q^{d+1}+q^{-d-1}),
\\
(a+a^{-1}) (b+b^{-1})+(c+c^{-1}) (q^{d+1}+q^{-d-1}),
\end{align*}
respectively.
\end{enumerate} 
\end{prop}

The sufficient and necessary condition for the $\triangle_q$-module $V_d(a,b,c)$ as irreducible was given in \cite[Theorem 4.4]{Huang:2015}. Moreover, those irreducible $\triangle_q$-modules are all $(d+1)$-dimensional irreducible $\triangle_q$-modules up to isomorphism \cite[Theorem 4.7]{Huang:2015}. The statements are as follows:

\begin{thm}
[Theorem 4.4, \cite{Huang:2015}]
\label{thm:irr_UAW}
For any nonzero scalars $a,b,c\in \F$ and any integer $d\geq 0$, the $\triangle_q$-module $V_d(a,b,c)$ is irreducible if and only if $$
abc, 
a^{-1}bc, 
ab^{-1}c, 
abc^{-1}
\not\in 
\{q^{2i-d-1}\,|\, i=1,2,\ldots,d\}.
$$
\end{thm}

\begin{thm}
[Theorem 4.7, \cite{Huang:2015}]
\label{thm:onto_UAW}
Let $d\geq 0$ denote an integer. If $V$ is a $(d+1)$-dimensional irreducible $\triangle_q$-module, then there are nonzero $a,b,c\in \F$ such that the $\triangle_q$-module $V_d(a,b,c)$ is isomorphic to $V$.
\end{thm}

It was given the following sufficient and necessary conditions for $A,B,C$ as diagonalizable on $V_d(a,b,c)$.

\begin{lem}
[Lemma 5.1, \cite{Huang:2015}]
\label{lem:dia_UAW}
Let $d\geq 0$ denote an integer and let $a,b,c$ denote nonzero scalars in $\F$. If the $\triangle_q$-module $V_d(a,b,c)$ is irreducible then the following are equivalent:
\begin{enumerate}
\item $A$ {\rm (}resp. $B${\rm )}  {\rm (}resp. $C${\rm )} is diagonalizable on $V_d(a,b,c)$.

\item $A$ {\rm (}resp. $B${\rm )}  {\rm (}resp. $C${\rm )} is multiplicity-free on $V_d(a,b,c)$.

\item $a^2$  {\rm (}resp. $b^2${\rm )}  {\rm (}resp. $c^2${\rm )} is not among $q^{2d-2},q^{2d-4},\ldots,q^{2-2d}$.
\end{enumerate}
\end{lem}

\section{The conditions for $A,B,C$ as diagonalizable on even-dimensional irreducible $\H_q$-modules}\label{s:even}

In this section we will prove Theorem \ref{thm:diagonal} in the even-dimensional case. We begin by some facts concerning the even-dimensional irreducible $\H_q$-modules.

Define
\begin{align*}
c_i&=t_i+t_i^{-1}
\qquad 
\hbox{for all $i=0,1,2,3$}.
\end{align*}
By Definition \ref{defn:H} the element $c_i$ is central in $\H_q$ for $i=0,1,2,3$.

\begin{prop}
[Proposition 2.3, \cite{Huang:DAHAmodule}]
\label{prop:E}
Let $d\geq 1$ denote an odd integer. Assume that $k_0,k_1,k_2,k_3$ are nonzero scalars in $\F$ with 
$$
k_0^2=q^{-d-1}.
$$ 
Then there exists a $(d+1)$-dimensional $\H_q$-module $E(k_0,k_1,k_2,k_3)$ satisfying the following conditions:
\begin{enumerate}
\item There exists an $\F$-basis $\{v_i\}_{i=0}^d$ for $E(k_0,k_1,k_2,k_3)$ such that 
\begin{align*}
t_0 v_i 
&=
\left\{
\begin{array}{ll}
\textstyle
k_0^{-1} q^{-i} (1-q^i) (1-k_0^2 q^i) 
v_{i-1}
+
(
k_0+k_0^{-1}-k_0^{-1}q^{-i}
) 
v_i
\qquad
&\hbox{for $i=2,4,\ldots,d-1$},
\\
k_0^{-1} q^{-i-1}
(v_i-v_{i+1})
\qquad
&\hbox{for $i=1,3,\ldots,d-2$},
\end{array}
\right.
\\
t_0 v_0&=k_0 v_0,
\qquad
t_0 v_d=k_0v_d,
\\
t_1 v_i
&=
\left\{
\begin{array}{ll}
-k_1(1-q^i)(1-k_0^2 q^i)v_{i-1}
+k_1 v_i
+k_1^{-1} v_{i+1}
\qquad
&\hbox{for $i=2,4,\ldots,d-1$},
\\
k_1^{-1} v_i
\qquad
&\hbox{for $i=1,3,\ldots,d$},
\end{array}
\right.
\\
t_1 v_0 &=k_1 v_0+k_1^{-1} v_1,
\\
t_2 v_i 
&=
\left\{
\begin{array}{ll}
k_0^{-1} k_1^{-1} k_3^{-1} q^{-i-1}
(v_i-v_{i+1})
\qquad
&\hbox{for $i=0,2,\ldots,d-1$},
\\
\textstyle
\frac{(k_0 k_1 k_3 q^i-k_2)
(k_0 k_1 k_3 q^i- k_2^{-1})}
{k_0 k_1 k_3 q^i  }
v_{i-1}
+
(k_2+k_2^{-1}-
k_0^{-1} k_1^{-1} k_3^{-1} q^{-i}
) v_i
\qquad
&\hbox{for $i=1,3,\ldots,d$},
\end{array}
\right.
\\
t_3 v_i 
&=
\left\{
\begin{array}{ll}
k_3 v_i
\qquad
&\hbox{for $i=0,2,\ldots,d-1$},
\\
-k_3^{-1}(k_0k_1k_3q^i-k_2)
(k_0k_1k_3 q^i-k_2^{-1})
v_{i-1}
+k_3^{-1} v_i
+k_3 v_{i+1}
\qquad
&\hbox{for $i=1,3,\ldots,d-2$}.
\end{array}
\right.
\\
t_3 v_d &=
-k_3^{-1}(k_0k_1k_3q^d-k_2)
(k_0k_1k_3 q^d-k_2^{-1})
v_{d-1}
+k_3^{-1} v_d.
\end{align*}

\item The elements $c_0, c_1,c_2,c_3$ act on $E(k_0,k_1,k_2,k_3)$ as scalar multiplication by 
$$
k_0+k_0^{-1},
\quad 
k_1+k_1^{-1},
\quad 
k_2+k_2^{-1},
\quad 
k_3+k_3^{-1},
$$
respectively.
\end{enumerate}
\end{prop}

For convenience, we adopt the following conventions in the rest of this section: Let $d\geq 1$ denote an odd integer. Let $k_0,k_1,k_2,k_3$ denote nonzero scalars in $\F$ with $k_0^2=q^{-d-1}$. Let $\{v_i\}_{i=0}^d$ denote the 
$\F$-basis for $E(k_0,k_1,k_2,k_3)$ from Proposition \ref{prop:E}(i).

The sufficient and necessary condition for $E(k_0,k_1,k_2,k_3)$ as irreducible was given in \cite[Theorem 5.8]{Huang:DAHAmodule}. 
Moreover, all $(d+1)$-dimensional irreducible $\H_q$-modules are obtained by twisting those irreducible $\H_q$-modules $E(k_0,k_1,k_2,k_3)$ up to isomorphism. The statements are as follows:

\begin{thm}
[Theorem 5.8, \cite{Huang:DAHAmodule}]
\label{thm:irr_E}
The $\H_q$-module $E(k_0,k_1,k_2,k_3)$ is irreducible if and only if 
$$
k_0k_1k_2k_3, k_0k_1^{-1}k_2k_3, k_0k_1k_2^{-1}k_3, k_0k_1k_2 k_3^{-1}\not=q^{-i}
\qquad \hbox{for all $i=1,3,\ldots,d$}.
$$
\end{thm}

Let $\Z$ denote the additive group of integers. Recall that $\Z/4\Z$ is isomorphic to the cyclic group of order four.
Observe that there exists a unique $\Z/4\Z$-action on $\H_q$ such that each element of $\Z/4\Z$ acts on $\H_q$ as an $\F$-algebra automorphism in the following way:

\begin{table}[H]
\centering
\extrarowheight=3pt
\begin{tabular}{c|rrrr}
$\e\in \Z/4\Z$  &$t_0$ &$t_1$ &$t_2$ &$t_3$ 
\\

\midrule[1pt]

${0\pmod 4}$ &$t_0$  &$t_1$ &$t_2$ &$t_3$ 
\\
${1\pmod 4}$ &$t_1$ &$t_2$ &$t_3$ &$t_0$ 
\\
${2\pmod 4}$ &$t_2$ &$t_3$ &$t_0$ &$t_1$
\\
${3\pmod 4}$ &$t_3$ &$t_0$ &$t_1$ &$t_2$
\end{tabular}
\caption{The $\Z/4\Z$-action on $\H_q$}\label{Z/4Z-action}
\end{table}

Let $V$ denote an $\H_q$-module. For any $\F$-algebra automorphism $\e$ of $\H_q$ the notation 
$
V^\e
$ 
stands for the $\H_q$-module obtained by twisting the $\H_q$-module $V$ via $\e$.

\begin{thm}
[Theorem 6.1, \cite{Huang:DAHAmodule}]
\label{thm:onto_E}
If $V$ is a $(d+1)$-dimensional irreducible $\H_q$-module, then there exist an element $\e\in \Z/4\Z$ and nonzero scalars $k_0,k_1,k_2,k_3\in \F$ with $k_0^2=q^{-d-1}$ such that the $\H_q$-module $E(k_0,k_1,k_2,k_3)^\e$ is isomorphic to $V$.
\end{thm}

We recall a result concerning the isomorphism class of $E(k_0,k_1,k_2,k_3)$.

\begin{thm}
[Theorem 2.2, \cite{Huang:DAHAmodule}]
\label{thm:iso_E}
If the $\H_q$-module $E(k_0,k_1,k_2,k_3)$ is irreducible then the following hold:
\begin{enumerate}
\item the $\H_q$-module $E(k_0,k_1,k_2,k_3)$ is isomorphic to $E(k_0,k_1^{-1},k_2,k_3)$.

\item the $\H_q$-module $E(k_0,k_1,k_2,k_3)$ is isomorphic to $E(k_0,k_1,k_2^{-1},k_3)$.

\item the $\H_q$-module $E(k_0,k_1,k_2,k_3)$ is isomorphic to $E(k_0,k_1,k_2,k_3^{-1})$.
\end{enumerate}
\end{thm}

Using Proposition \ref{prop:E}(i) a routine calculation yields the following lemma:

\begin{lem}
\label{lem:det_E}
The determinants of $t_0,t_1,t_2,t_3$ on $E(k_0,k_1,k_2,k_3)$ are $q^{-d-1},1,1,1$, respectively.
\end{lem}

We develop the following discriminant to determine the element $\e\in \Z/4\Z$ and the nonzero scalars $k_0,k_1,k_2,k_3\in \F$ in Theorem \ref{thm:onto_E}.

\begin{thm}
\label{thm:onto2_E}
Suppose that $V$ is a $(d+1)$-dimensional irreducible $\H_q$-module. For any $\e\in \Z/4\Z$ and any nonzero scalars $k_0,k_1,k_2,k_3$ with $k_0^2=q^{-d-1}$, the following are equivalent:
\begin{enumerate}
\item The $\H_q$-module $E(k_0,k_1,k_2,k_3)^\e$ is isomorphic to $V$. 

\item The determinant of $t_0$  on $V^{-\e}$ is equal to $q^{-d-1}$ and $c_0,c_1,c_2,c_3$ act on $V^{-\e}$ as scalar multiplication by $k_0+k_0^{-1},k_1+k_1^{-1},k_2+k_2^{-1},k_3+k_3^{-1}$, respectively.
\end{enumerate}
\end{thm}
\begin{proof}
(i) $\Rightarrow$ (ii): By (i) the $\H_q$-module $E(k_0,k_1,k_2,k_3)$ is isomorphic to $V^{-\e}$. Hence (ii) follows by Proposition \ref{prop:E}(ii) and Lemma \ref{lem:det_E}.

(ii) $\Rightarrow$ (i): By Theorem \ref{thm:onto_E} there are an $\e'\in \Z/4\Z$ and nonzero $k_0',k_1',k_2',k_3'\in \F$ with $k_0'^{2}=q^{-d-1}$ such that the $\H_q$-module $E(k_0',k_1',k_2',k_3')$ is isomorphic to $V^{-\e'}$. By Lemma \ref{lem:det_E} and since $q$ is not a root of unity, it follows that $\e=\e'$.
It follows from Proposition \ref{prop:E}(ii) that $k_0'=k_0$ and $k_i'=k_i^{\pm 1}$ for all $i=1,2,3$. 
 Hence (i) follows by Theorem \ref{thm:iso_E}.
\end{proof}

\begin{lem}
\label{lem:t0t1_E}
The action of $t_0t_1$ on $E(k_0,k_1,k_2,k_3)$ is as follows:
\begin{align*}
(1-k_0 k_1 q^{2\lceil \frac{i}{2}\rceil} (t_0 t_1)^{(-1)^{i-1}})
v_i
=\left\{
\begin{array}{ll}
v_{i+1} 
\qquad 
&\hbox{for $i=0,1,\ldots,d-1$},
\\
0
\qquad 
&\hbox{for $i=d$}.
\end{array}
\right.
\end{align*}
\end{lem}
\begin{proof}
Apply Proposition \ref{prop:E} to evaluate the action of $t_0t_1$ on $E(k_0,k_1,k_2,k_3)$. 
\end{proof}

\begin{lem}\label{lem:t0t2_E}
If the $\H_q$-module $E(k_0,k_1,k_2,k_3)$ is irreducible then the following hold:
\begin{enumerate}
\item There exists an $\F$-basis $\{w_i\}_{i=0}^d$ for $E(k_0,k_1,k_2,k_3)$ such that 
\begin{align*}
(1-k_0 k_2 q^{2\lceil \frac{i}{2}\rceil} (t_2 t_0)^{(-1)^{i-1}})
w_i
=\left\{
\begin{array}{ll}
w_{i+1} 
\qquad 
&\hbox{for $i=0,1,\ldots,d-1$},
\\
0
\qquad 
&\hbox{for $i=d$}.
\end{array}
\right.
\end{align*}

\item There exists an $\F$-basis $\{w_i\}_{i=0}^d$ for  $E(k_0,k_1,k_2,k_3)$ such that 
\begin{align*}
(1-k_0 k_3 q^{2\lceil \frac{i}{2}\rceil} (t_3 t_0)^{(-1)^{i-1}})
w_i
=\left\{
\begin{array}{ll}
w_{i+1} 
\qquad 
&\hbox{for $i=0,1,\ldots,d-1$},
\\
0
\qquad 
&\hbox{for $i=d$}.
\end{array}
\right.
\end{align*}

\item There exists an $\F$-basis $\{w_i\}_{i=0}^d$ for  $E(k_0,k_1,k_2,k_3)$ such that 
\begin{align*}
(1-k_0 k_2 q^{2\lfloor \frac{i}{2}\rfloor+1} (t_1 t_3)^{(-1)^{i}})
w_i
=\left\{
\begin{array}{ll}
w_{i+1} 
\qquad 
&\hbox{for $i=0,1,\ldots,d-1$},
\\
0
\qquad 
&\hbox{for $i=d$}.
\end{array}
\right.
\end{align*}
\end{enumerate}
\end{lem}
\begin{proof}
Suppose the $\H_q$-module $V=E(k_0,k_1,k_2,k_3)$ is irreducible.

(i): By Definition \ref{defn:H} there exists a unique $\F$-algebra automorphism $\rho:\H_q\to \H_q$ given by 
\begin{eqnarray}\label{auto_t0t2}
(t_0,t_1,t_2,t_3) 
&\mapsto &
(t_0,t_0^{-1} t_2t_0,t_0^{-1} t_3t_0,t_1)
\end{eqnarray}
whose inverse sends 
$(t_0,t_1,t_2,t_3)$ to
$(t_0,t_3 ,t_0 t_1t_0^{-1},t_0 t_2t_0^{-1})$. By Proposition \ref{prop:E}(ii) the elements $c_0,c_1,c_2,c_3$ act on $V^\rho$ as scalar multiplication by 
$
k_0+k_0^{-1},k_2+k_2^{-1},k_3+k_3^{-1},k_1+k_1^{-1}$, 
respectively.
By Lemma \ref{lem:det_E} the determinant of $t_0$ on $V^\rho$ is $q^{-d-1}$. Therefore the $\H_q$-module $V^\rho$ is isomorphic to 
$
E(k_0,k_2,k_3,k_1)
$
by Theorem \ref{thm:onto2_E}. 
It follows from Lemma \ref{lem:t0t1_E} that there exists an $\F$-basis $\{w_i\}_{i=0}^d$ for $V^\rho$ such that 
\begin{align*}
(1-k_0 k_2 q^{2\lceil \frac{i}{2}\rceil} (t_0 t_1)^{(-1)^{i-1}})
w_i
=\left\{
\begin{array}{ll}
w_{i+1} 
\qquad 
&\hbox{for $i=0,1,\ldots,d-1$},
\\
0
\qquad 
&\hbox{for $i=d$}.
\end{array}
\right.
\end{align*}
It follows from (\ref{auto_t0t2}) that the action of $t_0t_1$ on $V^\rho$ is identical to the action of $t_2 t_0$ on $V$. Hence (i) follows.

(ii): By Definition \ref{defn:H} there exists a unique $\F$-algebra automorphism $\rho:\H_q\to \H_q$ given by 
\begin{eqnarray}\label{auto_t3t0}
(t_0,t_1,t_2,t_3) 
&\mapsto &
(t_0,t_0^{-1} t_3 t_0,t_1 t_2 t_1^{-1}, t_1)
\end{eqnarray}
whose inverse sends 
$(t_0,t_1,t_2,t_3)$ to
$(t_0,t_3,t_3^{-1} t_2 t_3,t_0 t_1 t_0^{-1})$. By Proposition \ref{prop:E}(ii) the elements $c_0,c_1,c_2,c_3$ act on $V^\rho$ as scalar multiplication by 
$
k_0+k_0^{-1},k_3+k_3^{-1},k_2+k_2^{-1},k_1+k_1^{-1}$, 
respectively.
By Lemma \ref{lem:det_E} the determinant of $t_0$ on $V^\rho$ is $q^{-d-1}$. Therefore the $\H_q$-module $V^\rho$ is isomorphic to 
$
E(k_0,k_3,k_2,k_1)
$
by Theorem \ref{thm:onto2_E}.
It follows from Lemma \ref{lem:t0t1_E} that there exists an $\F$-basis $\{w_i\}_{i=0}^d$ for $V^\rho$ such that 
\begin{align*}
(1-k_0 k_3 q^{2\lceil \frac{i}{2}\rceil} (t_0 t_1)^{(-1)^{i-1}})
w_i
=\left\{
\begin{array}{ll}
w_{i+1} 
\qquad 
&\hbox{for $i=0,1,\ldots,d-1$},
\\
0
\qquad 
&\hbox{for $i=d$}.
\end{array}
\right.
\end{align*}
It follows from (\ref{auto_t3t0}) that the action of $t_0t_1$ on $V^\rho$ is identical to the action of $t_3 t_0$ on $V$. Hence (ii) follows.

(iii):  By Definition \ref{defn:H} there exists a unique $\F$-algebra automorphism $\rho:\H_q\to \H_q$ given by 
\begin{eqnarray}\label{auto_t1t3}
(t_0,t_1,t_2,t_3) 
&\mapsto &
(t_1 t_2 t_1^{-1},t_1,t_3,t_0)
\end{eqnarray}
whose inverse sends 
$(t_0,t_1,t_2,t_3)$ to
$(t_3,t_1,t_1^{-1} t_0 t_1,t_2)$. By Proposition \ref{prop:E}(ii) the elements $c_0,c_1,c_2,c_3$ act on $V^\rho$ as scalar multiplication by 
$
k_2+k_2^{-1},k_1+k_1^{-1},k_3+k_3^{-1},k_0+k_0^{-1}$, 
respectively.
By Lemma \ref{lem:det_E} the determinant of $t_3$ on $V^\rho$ is $q^{-d-1}$. Therefore the $\H_q$-module $V^\rho$ is isomorphic to 
$
E(k_0,k_2,k_1,k_3)^{1\bmod{4}}
$ 
by Theorem \ref{thm:onto2_E}. 
Let $\{w_i\}_{i=0}^d$ denote the basis for $E(k_0,k_2,k_1,k_3)$ from Proposition \ref{prop:E}(i).
Since $t_0 t_1=(q t_2 t_3)^{-1}$ it follows from Lemma \ref{lem:t0t1_E} that 
\begin{align*}
(1-k_0 k_2 q^{2\lfloor \frac{i}{2}\rfloor+1} (t_2 t_3)^{(-1)^{i}})
w_i
=\left\{
\begin{array}{ll}
w_{i+1} 
\qquad 
&\hbox{for $i=0,1,\ldots,d-1$},
\\
0
\qquad 
&\hbox{for $i=d$}.
\end{array}
\right.
\end{align*}
By Table \ref{Z/4Z-action} the action of $t_1 t_2$ on $E(k_0,k_2,k_1,k_3)^{1\bmod{4}}$ is identical to the action of $t_2 t_3$ on $E(k_0,k_2,k_1,k_3)$. 
By (\ref{auto_t1t3}) the action of $t_1t_2$ on $V^\rho$ is identical to the action of $t_1 t_3$ on $V$. By the above comments the statement (iii) follows.
\end{proof}

Using Lemmas \ref{lem:t0t1_E} and \ref{lem:t0t2_E} yields the following equivalent conditions for 
$t_it_0$ and $t_0 t_i$
as diagonalizable on $E(k_0,k_1,k_2,k_3)$ for $i\in\{1,2,3\}$. Additionally, we obtain the following sufficient conditions for 
$
t_i t_j+(t_j t_i)^{-1}
$
as diagonalizable on $E(k_0,k_1,k_2,k_3)$ for distinct $i,j\in \{1,2,3\}$.

\begin{lem}\label{lem:t1t0_diag}
If the $\H_q$-module $E(k_0,k_1,k_2,k_3)$ is irreducible the the following are equivalent:
\begin{enumerate}
\item $t_1 t_0$ and $t_0 t_1$ are diagonalizable on $E(k_0,k_1,k_2,k_3)$.

\item $t_1 t_0$ and $t_0 t_1$ are multiplicity-free on $E(k_0,k_1,k_2,k_3)$.

\item $k_1^2$ is not among $q^{d-1},q^{d-3},\ldots,q^{1-d}$.
\end{enumerate}
\end{lem}
\begin{proof}
Note that $t_1 t_0$ is similar to $t_0 t_1$ on $E(k_0,k_1,k_2,k_3)$.

(ii) $\Leftrightarrow$ (iii): 
By Lemma \ref{lem:t0t1_E} this characteristic polynomial of $t_0 t_1$ in $E(k_0,k_1,k_2,k_3)$ has the roots
\begin{gather}\label{eigen_t0t1_E}
k_0 k_1 q^i 
\qquad \hbox{for $i=0,2,\ldots,d-1$};
\qquad 
k_0^{-1} k_1^{-1} q^{-i-1} 
\qquad \hbox{for $i=1,3,\ldots,d$}.
\end{gather}
Since  $q$ is not a root of unity and $k_0^2=q^{-d-1}$, the scalars (\ref{eigen_t0t1_E}) are mutually distinct if and only if (iii) holds. Therefore (ii) and (iii) are equivalent.

(ii), (iii) $\Rightarrow$ (i): Trivial.

(i) $\Rightarrow$ (ii), (iii): By Lemma \ref{lem:t0t1_E} the product 
\begin{gather*}
\prod_{i=0}^d (1-k_0 k_1 q^{2\lceil \frac{i}{2}\rceil} (t_0 t_1)^{(-1)^{i-1}})
\end{gather*}
vanishes at $v_0$. Moreover $(1-k_0 k_1 q^{2\lceil \frac{i}{2}\rceil} (t_0 t_1)^{(-1)^{i-1}})v_j$ is equal to 
\begin{align*}
\left\{
\begin{array}{ll}
(1-q^{i-j})v_j+q^{i-j} v_{j+1}
\qquad 
&\hbox{if $i=j\bmod{2}$},
\\
(1-k_0^2 k_1^2 q^{i+j+1}) v_j
-q^{i-j-1}(v_{j+1}-v_{j+2})
\qquad 
&\hbox{if $i\not= j\bmod{2}$}
\end{array}
\right.
\end{align*}
for all $i,j=0,1,\ldots,d$, where $v_{d+1}=0$ and $v_{d+2}=0$.
Using the above result yields that
\begin{gather*}
\prod_{\substack{i=0 \\ i\not=j}}^d (1-k_0 k_1 q^{2\lceil \frac{i}{2}\rceil} (t_0 t_1)^{(-1)^{i-1}}) v_0\not=0
\end{gather*}
for all $j=0,1,\ldots,d$. More precisely it is equal to $(-1)^{j+1}v_d$ plus a linear combination of $v_0,v_1,\ldots,v_{d-1}$. Hence $t_0 t_1$-annihilator of $v_0$ is equal to the characteristic polynomial of $t_0 t_1$ in $E(k_0,k_1,k_2,k_3)$. Therefore (i) implies (ii) and (iii).
\end{proof}

\begin{lem}\label{lem:t2t0_diag}
If the $\H_q$-module $E(k_0,k_1,k_2,k_3)$ is irreducible the the following are equivalent:
\begin{enumerate}
\item $t_2 t_0$ and $t_0 t_2$ are diagonalizable on $E(k_0,k_1,k_2,k_3)$.

\item $t_2 t_0$ and $t_0 t_2$ are multiplicity-free on $E(k_0,k_1,k_2,k_3)$.

\item $k_2^2$ is not among $q^{d-1},q^{d-3},\ldots,q^{1-d}$.
\end{enumerate}
\end{lem}
\begin{proof}
Applying Lemma \ref{lem:t0t2_E}(i) the lemma follows by an argument similar to the proof of Lemma \ref{lem:t1t0_diag}.
\end{proof}

\begin{lem}\label{lem:t3t0_diag}
If the $\H_q$-module $E(k_0,k_1,k_2,k_3)$ is irreducible the the following are equivalent:
\begin{enumerate}
\item $t_3 t_0$ and $t_0 t_3$ are diagonalizable on $E(k_0,k_1,k_2,k_3)$.

\item $t_3 t_0$ and $t_0 t_3$ are multiplicity-free on $E(k_0,k_1,k_2,k_3)$.

\item $k_3^2$ is not among $q^{d-1},q^{d-3},\ldots,q^{1-d}$.
\end{enumerate}
\end{lem}
\begin{proof}
Applying Lemma \ref{lem:t0t2_E}(ii) the lemma follows by an argument similar to the proof of Lemma \ref{lem:t1t0_diag}.
\end{proof}

\begin{lem}\label{lem:t1t3+t1t3inv_diag}
If the $\H_q$-module $E(k_0,k_1,k_2,k_3)$ is irreducible then the following hold:
\begin{enumerate}

\item If $k_1^2$ is not among $q^{d-3},q^{d-5},\ldots,q^{3-d}$, then 
$
t_2 t_3+(t_2 t_3)^{-1}
$
and 
$
t_3 t_2+(t_3 t_2)^{-1}
$ 
is diagonalizable on $E(k_0,k_1,k_2,k_3)$.

\item If $k_2^2$ is not among $q^{d-3},q^{d-5},\ldots,q^{3-d}$, then 
$
t_1 t_3+(t_1 t_3)^{-1}
$ 
and 
$
t_3 t_1+(t_3 t_1)^{-1}
$
is diagonalizable on $E(k_0,k_1,k_2,k_3)$.

\item If $k_3^2$ is not among $q^{d-3},q^{d-5},\ldots,q^{3-d}$, then 
$
t_1 t_2+(t_1 t_2)^{-1}
$
and  
$
t_2 t_1+(t_2 t_1)^{-1}
$
is diagonalizable on $E(k_0,k_1,k_2,k_3)$.

\end{enumerate}
\end{lem}
\begin{proof}
(i): Let $D=t_2 t_3+(t_2 t_3)^{-1}$. 
Note that $D$ is similar to $t_3 t_2+(t_3 t_2)^{-1}$ on $E(k_0,k_1,k_2,k_3)$. More precisely, 
$
t_i t_j+(t_i t_j)^{-1}=t_j t_i+(t_j t_i)^{-1}
$
for all $i,j\in\{0,1,2,3\}$
by \cite[Lemma 3.8]{daha&Z3}. 
Thus it suffices to show that $D$ is diagonalizable on $E(k_0,k_1,k_2,k_3)$.
Using (\ref{t0123}) yields that $t_2 t_3=(q t_0 t_1)^{-1}$. Hence 
\begin{gather*}
D=q t_0 t_1+(qt_0 t_1)^{-1}.
\end{gather*}
Given any element $X$ of $\H_q$, let $[X]$ denote the matrix representing $X$ with respect to $\{v_i\}_{i=0}^d$. 
Using Lemma \ref{lem:t0t1_E} a direct calculation yields that $[D]$ is a lower triangular matrix of the form
\begin{gather}\label{[D]}
\begin{pmatrix}
\theta_0 & & & & & &{\bf 0}
\\
0 &\theta_0
\\
  & &\theta_1
\\
  &  & &\theta_1
\\
  & & & &\ddots
\\
  & & & & &\theta_{\frac{d-1}{2}}
\\  
*  & & & & &0 &\theta_{\frac{d-1}{2}}
\end{pmatrix}
\end{gather}
where 
$$
\theta_i=k_0 k_1 q^{2i+1}+k_0^{-1} k_1^{-1} q^{-2i-1}
\qquad 
\hbox{for $i=0,1,\ldots,\frac{d-1}{2}$}.
$$

Since $q$ is not a root of unity and $k_1^2$ is not among $q^{3-d},q^{5-d},\ldots,q^{d-3}$, the scalars $\{\theta_i\}_{i=0}^{\frac{d-1}{2}}$ are mutually distinct. Hence the $\theta_i$-eigenspace of $D$ in $E(k_0,k_1,k_2,k_3)$ has dimension less than or equal to two for all $i=0,1,\ldots,\frac{d-1}{2}$. By (\ref{[D]}) the first two rows of $[D-\theta_0]$ are zero and the last two columns of $[D-\theta_{\frac{d-1}{2}}]$ are zero. By the rank-nullity theorem, the $\theta_i$-eigenspace of $D$ in $E(k_0,k_1,k_2,k_3)$ has dimension two for $i=0,\frac{d-1}{2}$.
To see the diagonalizability of $D$ it remains to show that the $\theta_i$-eigenspace of $D$ in $E(k_0,k_1,k_2,k_3)$ has dimension two for all $i=1,2,\ldots,\frac{d-3}{2}$.

By Lemma \ref{lem:t0t1_E} the matrix $q [t_0 t_1]$ is a lower triangular matrix of the form 
\begin{gather}\label{[t1t3]}
\begin{pmatrix}
\vartheta_0 & & & & & &{\bf 0}
\\
 &\vartheta_0^{-1}
\\
  & &\vartheta_1
\\
  &  & &\vartheta_1^{-1}
\\
  & & & &\ddots
\\
  & & & & &\vartheta_{\frac{d-1}{2}}
\\  
*  & & & & & &\vartheta_{\frac{d-1}{2}}^{-1}
\end{pmatrix}
\end{gather}
where 
$$
\vartheta_i=k_0 k_1 q^{2i+1}
\qquad 
\hbox{for $i=0,1,\ldots,\frac{d-1}{2}$}.
$$
Since $q$ is not a root of unity and $k_1^2$ is not among $q^{3-d},q^{5-d},\ldots,q^{d-3}$, the eigenvalues $\{\vartheta_i^{\pm 1}\}_{i=1}^\frac{d-3}{2}$ of $q t_0 t_1$ are multiplicity-free.
Hence the $\theta_i$-eigenspace of $D$ in $E(k_0,k_1,k_2,k_3)$ has dimension two for all $i=1,2,\ldots,\frac{d-3}{2}$. The statement (i) follows.

(ii): Using Lemma \ref{lem:t0t2_E}(iii) the statement (ii) follows by an argument similar to the part (i).

(iii): Using Lemma \ref{lem:t0t2_E}(ii) the statement (iii) follows by an argument similar to the part (i). 
\end{proof}

Recall the finite-dimensional irreducible $\triangle_q$-modules from \S\ref{s:AWmodule}. The composition factors of any $(d+1)$-dimensional irreducible $\H_q$-modules are classified as follows.

\begin{thm}
[\S 4.2--\S 4.5, \cite{Huang:AW&DAHAmodule}]
\label{thm:CF_even}
If the $\H_q$-module $E(k_0,k_1,k_2,k_3)$ is irreducible then the following hold:
\begin{enumerate}
\item If $d=1$ then the $\triangle_q$-module $E(k_0,k_1,k_2,k_3)$ is irreducible and it is isomorphic to 
\begin{align*}
V_1(k_0k_1q,k_0 k_3 q, k_0 k_2 q).
\end{align*}

\item If $d\geq 3$ then the factors of any composition series for the $\triangle_q$-module $E(k_0,k_1,k_2,k_3)$ are isomorphic to 
\begin{align*}
&V_\frac{d+1}{2}(k_0k_1q^{\frac{d+1}{2}},k_0 k_3 q^{\frac{d+1}{2}}, k_0 k_2 q^{\frac{d+1}{2}}),
\\
&V_\frac{d-3}{2}(k_0k_1q^{\frac{d+1}{2}},k_0 k_3 q^{\frac{d+1}{2}}, k_0 k_2 q^{\frac{d+1}{2}}).
\end{align*}

\item The factors of any composition series for the $\triangle_q$-module $E(k_0,k_1,k_2,k_3)^{1\bmod 4}$ are isomorphic to
\begin{align*}
&V_\frac{d-1}{2}(k_0 k_3 q^{\frac{d+1}{2}}, k_0k_1q^{\frac{d+3}{2}},k_0 k_2 q^{\frac{d+1}{2}}),
\\
&V_\frac{d-1}{2}(k_0 k_3 q^{\frac{d+1}{2}}, k_0k_1q^{\frac{d-1}{2}},k_0 k_2 q^{\frac{d+1}{2}}).
\end{align*}

\item The factors of any composition series for the $\triangle_q$-module $E(k_0,k_1,k_2,k_3)^{2\bmod 4}$ are isomorphic to
\begin{align*}
&V_\frac{d-1}{2}(k_0k_1q^{\frac{d+1}{2}},k_0 k_3 q^{\frac{d+1}{2}}, k_0 k_2 q^{\frac{d+3}{2}}),
\\
&V_\frac{d-1}{2}(k_0k_1q^{\frac{d+1}{2}},k_0 k_3 q^{\frac{d+1}{2}}, k_0 k_2 q^{\frac{d-1}{2}}).
\end{align*}

\item The factors of any composition series for the $\triangle_q$-module $E(k_0,k_1,k_2,k_3)^{3\bmod 4}$ are isomorphic to
\begin{align*}
&V_\frac{d-1}{2}(k_0 k_3 q^{\frac{d-1}{2}}, k_0k_1q^{\frac{d+1}{2}},k_0 k_2 q^{\frac{d+1}{2}}),
\\
&V_\frac{d-1}{2}(k_0 k_3 q^{\frac{d+3}{2}}, k_0k_1q^{\frac{d+1}{2}},k_0 k_2 q^{\frac{d+1}{2}}).
\end{align*}
\end{enumerate}
\end{thm}

Applying Lemma \ref{lem:dia_UAW} to Theorem \ref{thm:CF_even} yields the following lemmas:

\begin{lem}\label{lem:diag_ABC_E}
If the $\H_q$-module $E(k_0,k_1,k_2,k_3)$ is irreducible then the following hold:
\begin{enumerate}
\item $A$ 
is multiplicity-free on all composition factors of the $\triangle_q$-module $E(k_0,k_1,k_2,k_3)$ if and only if $k_1^2$ is not among $q^{d-1},q^{d-3},\ldots,q^{1-d}$.

\item $B$
is multiplicity-free on all composition factors of the $\triangle_q$-module $E(k_0,k_1,k_2,k_3)$ if and only if $k_3^2$ is not among $q^{d-1},q^{d-3},\ldots,q^{1-d}$.

\item $C$
is multiplicity-free on all composition factors of the $\triangle_q$-module $E(k_0,k_1,k_2,k_3)$ if and only if $k_2^2$ is not among $q^{d-1},q^{d-3},\ldots,q^{1-d}$.
\end{enumerate}
\end{lem}

\begin{lem}\label{lem:diag_ABC_E1}
If the $\H_q$-module $E(k_0,k_1,k_2,k_3)$ is irreducible then the following hold:
\begin{enumerate}
\item $A$ 
is multiplicity-free on all composition factors of the $\triangle_q$-module $E(k_0,k_1,k_2,k_3)^{1\bmod{4}}$ if and only if 
$k_3^2$ is not among $q^{d-3},q^{d-5},\ldots,q^{3-d}$.

\item $B$
is multiplicity-free on all composition factors of the $\triangle_q$-module $E(k_0,k_1,k_2,k_3)^{1\bmod{4}}$ if and only if 
$k_1^2$ is not among $q^{d-1},q^{d-3},\ldots,q^{1-d}$.

\item $C$ 
is multiplicity-free on all composition factors of the $\triangle_q$-module $E(k_0,k_1,k_2,k_3)^{1\bmod{4}}$ if and only if 
$k_2^2$ is not among $q^{d-3},q^{d-5},\ldots,q^{3-d}$.
\end{enumerate}
\end{lem}

\begin{lem}\label{lem:diag_ABC_E2}
If the $\H_q$-module $E(k_0,k_1,k_2,k_3)$ is irreducible then the following hold:
\begin{enumerate}
\item $A$ 
is multiplicity-free on all composition factors of the $\triangle_q$-module $E(k_0,k_1,k_2,k_3)^{2\bmod{4}}$ if and only if 
$k_1^2$ is not among $q^{d-3},q^{d-5},\ldots,q^{3-d}$.

\item $B$
is multiplicity-free on all composition factors of the $\triangle_q$-module $E(k_0,k_1,k_2,k_3)^{2\bmod{4}}$ if and only if 
$k_3^2$ is not among $q^{d-3},q^{d-5},\ldots,q^{3-d}$.

\item $C$ 
is multiplicity-free on all composition factors of the $\triangle_q$-module $E(k_0,k_1,k_2,k_3)^{2\bmod{4}}$ if and only if 
$k_2^2$ is not among $q^{d-1},q^{d-3},\ldots,q^{1-d}$.
\end{enumerate}
\end{lem}

\begin{lem}\label{lem:diag_ABC_E3}
If the $\H_q$-module $E(k_0,k_1,k_2,k_3)$ is irreducible then the following hold:
\begin{enumerate}
\item $A$ 
is multiplicity-free on all composition factors of the $\triangle_q$-module $E(k_0,k_1,k_2,k_3)^{3\bmod{4}}$ if and only if 
$k_3^2$ is not among $q^{d-1},q^{d-3},\ldots,q^{1-d}$.

\item $B$
is multiplicity-free on all composition factors of the $\triangle_q$-module $E(k_0,k_1,k_2,k_3)^{3\bmod{4}}$ if and only if 
$k_1^2$ is not among $q^{d-3},q^{d-5},\ldots,q^{3-d}$.

\item $C$ 
is multiplicity-free on all composition factors of the $\triangle_q$-module $E(k_0,k_1,k_2,k_3)^{3\bmod{4}}$ if and only if 
$k_2^2$ is not among $q^{d-3},q^{d-5},\ldots,q^{3-d}$.
\end{enumerate}
\end{lem}

We are in the position to prove Theorem \ref{thm:diagonal} in the even-dimensional case.

\begin{thm}\label{thm:diagonal_even}
If $V$ is an even-dimensional irreducible $\H_q$-module then the following are equivalent:
\begin{enumerate}
\item $A$ {\rm (}resp. $B${\rm )} {\rm (}resp. $C${\rm )}
is diagonalizable on $V$.

\item $A$ {\rm (}resp. $B${\rm )} {\rm (}resp. $C${\rm )} is diagonalizable on all composition factors of the $\triangle_q$-module $V$.

\item $A$ {\rm (}resp. $B${\rm )} {\rm (}resp. $C${\rm )} is multiplicity-free on all composition factors of the $\triangle_q$-module $V$.
\end{enumerate}
\end{thm}
\begin{proof}
(i) $\Rightarrow$ (ii): Trivial.

(ii) $\Rightarrow$ (iii): Immediate from Lemma \ref{lem:dia_UAW}.

(iii) $\Rightarrow$ (i): 
Suppose that (iii) holds. 
Let $d=\dim V-1$. 
By Theorem \ref{thm:onto_E} there are an 
$\e\in \Z/4\Z$ and nonzero scalars $k_0,k_1,k_2,k_3\in \F$ with $k_0^2=q^{-d-1}$ such that the $\H_q$-module $E(k_0,k_1,k_2,k_3)^\e$ is isomorphic to $V$.

Using Lemmas \ref{lem:t1t0_diag}--\ref{lem:t3t0_diag} and \ref{lem:diag_ABC_E} yields that (i) holds for the case $\e=0\pmod{4}$. Note that the actions of $A,B,C$ on $E(k_0,k_1,k_2,k_3)^{1\bmod{4}}$ are identical to the actions of 
$$
t_2 t_1+(t_2 t_1)^{-1},
\qquad 
t_0 t_1+(t_0 t_1)^{-1},
\qquad 
t_3 t_1+(t_3 t_1)^{-1}
$$ 
on $E(k_0,k_1,k_2,k_3)$, respectively. Hence (i) holds for the case $\e=1\pmod{4}$ by Lemmas  \ref{lem:t1t0_diag}, \ref{lem:t1t3+t1t3inv_diag} and \ref{lem:diag_ABC_E1}. 
The actions of $A,B,C$ on $E(k_0,k_1,k_2,k_3)^{2\bmod{4}}$ are identical to the actions of 
$$
t_3 t_2+(t_3 t_2)^{-1},
\qquad 
t_1 t_2+(t_1 t_2)^{-1},
\qquad 
t_0 t_2+(t_0 t_2)^{-1}
$$ 
on $E(k_0,k_1,k_2,k_3)$, respectively. Hence (i) holds for the case $\e=2\pmod{4}$ by Lemmas  \ref{lem:t2t0_diag}, \ref{lem:t1t3+t1t3inv_diag} and \ref{lem:diag_ABC_E2}. The actions of $A,B,C$ on $E(k_0,k_1,k_2,k_3)^{3\bmod{4}}$ are identical to the actions of 
$$
t_0 t_3+(t_0 t_3)^{-1},
\qquad 
t_2 t_3+(t_2 t_3)^{-1},
\qquad 
t_1 t_3+(t_1 t_3)^{-1}
$$ 
on $E(k_0,k_1,k_2,k_3)$, respectively. 
Hence (i) holds for the case $\e=3\pmod{4}$ by Lemmas  \ref{lem:t3t0_diag}, \ref{lem:t1t3+t1t3inv_diag} and \ref{lem:diag_ABC_E3}. 

We have shown that (i) holds for all elements $\e\in \Z/4\Z$. Therefore (i) follows.
\end{proof}

We end this section with an example to show that the condition \ref{XY} is not a necessary condition for $A$ and $B$ as diagonalizable on even-dimensional irreducible $\H_q$-modules.

\begin{exam}\label{exam:b_E}
Set 
$
(k_0,k_1,k_3)=
(q^{-\frac{d+1}{2}},
q^{\frac{d-1}{2}},
q^{\frac{d-1}{2}})
$
and choose $k_2$ as a nonzero scalar in $\F$ satisfying 
$$
k_2\not\in\{q^\frac{3d-3}{2},q^\frac{3d-7}{2},\ldots,q^\frac{3-3d}{2}\}.
$$
By Theorem \ref{thm:irr_E} the $\H_q$-module $E(k_0,k_1,k_2,k_3)$ is irreducible.
The actions of $t_0 t_1$ and $t_3 t_0$ on $E(k_0,k_1,k_2,k_3)^{2\bmod{4}}$ are identical to the actions of 
$t_2 t_3$ and $t_1 t_2$
on $E(k_0,k_1,k_2,k_3)$, respectively. Since $t_2 t_3=(qt_0 t_1)^{-1}$ by (\ref{t0123}) it follows from Lemma \ref{lem:t1t0_diag} that $t_0 t_1$ is not diagonalizable on $E(k_0,k_1,k_2,k_3)^{2\bmod{4}}$.
Since $t_1 t_2=(qt_3 t_0)^{-1}$ by (\ref{t0123}) it follows from Lemma \ref{lem:t3t0_diag} that $t_3 t_0$ is not diagonalizable on $E(k_0,k_1,k_2,k_3)^{2\bmod{4}}$. 
However, the elements $A$ and $B$ are diagonalizable on $E(k_0,k_1,k_2,k_3)^{2\bmod{4}}$ by Lemma \ref{lem:t1t3+t1t3inv_diag}.
\end{exam}

\section{The conditions for $A,B,C$ as diagonalizable on odd-dimensional irreducible $\H_q$-modules}\label{s:odd}

In this section we will prove Theorem \ref{thm:diagonal} in the odd-dimensional case. We begin by some facts concerning the odd-dimensional irreducible $\H_q$-modules.

\begin{prop}
[Proposition 2.6, \cite{Huang:DAHAmodule}]
\label{prop:O}
Let $d\geq 0$ denote an even integer. Assume that $k_0,k_1,k_2,k_3$ are nonzero scalars in $\F$ with 
$$
k_0 k_1 k_2 k_3=q^{-d-1}.
$$
Then there exists a $(d+1)$-dimensional $\H_q$-module
 $O(k_0,k_1,k_2,k_3)$ satisfying the following conditions:
\begin{enumerate} 
\item There exists an $\F$-basis $\{v_i\}_{i=0}^d$ for $O(k_0,k_1,k_2,k_3)$  such that 
\begin{align*}
t_0 v_i 
&=
\left\{
\begin{array}{ll}
\textstyle
k_0^{-1} q^{-i} (1-q^i) (1-k_0^2 q^i) 
v_{i-1}
+
(
k_0+k_0^{-1}-k_0^{-1}q^{-i}
) 
v_i
\qquad
&\hbox{for $i=2,4,\ldots,d$},
\\
\textstyle
k_0^{-1} q^{-i-1}
(v_i-v_{i+1})
\qquad
&\hbox{for $i=1,3,\ldots,d-1$},
\end{array}
\right.
\\
t_0 v_0&= k_0 v_0,
\\
t_1 v_i
&=
\left\{
\begin{array}{ll}
-k_1(1-q^i)(1-k_0^2 q^i)v_{i-1}
+k_1 v_i
+k_1^{-1} v_{i+1}
\qquad
&\hbox{for $i=2,4,\ldots,d-2$},
\\
k_1^{-1} v_i
\qquad
&\hbox{for $i=1,3,\ldots,d-1$},
\end{array}
\right.
\\
t_1 v_0&= k_1 v_0 +k_1^{-1} v_1,
\qquad 
t_1 v_d=-k_1(1-q^d)(1-k_0^2 q^d)v_{d-1}
+k_1 v_d,
\\
t_2 v_i 
&=
\left\{
\begin{array}{ll}
k_2 q^{d-i}
(v_i-v_{i+1})
\qquad
&\hbox{for $i=0,2,\ldots,d-2$},
\\
\textstyle
-k_2(1-k_2^{-2}q^{i-d-1})
(1-q^{d-i+1})
v_{i-1}
+
(k_2+k_2^{-1}-
k_2 q^{d-i+1}) v_i
\qquad
&\hbox{for $i=1,3,\ldots,d-1$},
\end{array}
\right.
\\
t_2 v_d &= k_2 v_d,
\\
t_3 v_i 
&=
\left\{
\begin{array}{ll}
k_3 v_i
\qquad
&\hbox{for $i=0,2,\ldots,d$},
\\
\textstyle
-k_3^{-1}
(1-k_2^{-2} q^{i-d-1})
(1-q^{i-d-1})
v_{i-1}
+k_3^{-1} v_i
+k_3 v_{i+1}
\qquad
&\hbox{for $i=1,3,\ldots,d-1$}.
\end{array}
\right.
\end{align*}
\item The elements $c_0, c_1,c_2,c_3$ act on $O(k_0,k_1,k_2,k_3)$ as scalar multiplication by 
$$
k_0+k_0^{-1},
\quad 
k_1+k_1^{-1},
\quad 
k_2+k_2^{-1},
\quad 
k_3+k_3^{-1},
$$
respectively.
\end{enumerate}
\end{prop}

For convenience, we adopt the following notations in the rest of this section: Let $d\geq 0$ denote an even integer. Let $k_0,k_1,k_2,k_3$ denote nonzero scalars in $\F$ with $k_0 k_1 k_2 k_3=q^{-d-1}$. Let $\{v_i\}_{i=0}^d$ denote the $\F$-basis for $O(k_0,k_1,k_2,k_3)$ from Proposition \ref{prop:O}(i).

The sufficient and necessary condition for the $\H_q$-module $O(k_0,k_1,k_2,k_3)$ as irreducible was given in \cite[Theorem 7.7]{Huang:DAHAmodule}. Moreover, those irreducible $\H_q$-modules are all $(d+1)$-dimensional irreducible $\H_q$-modules up to isomorphism \cite[Theorem 8.1]{Huang:DAHAmodule}. The statements are as follows:

\begin{thm}
[Theorem 7.7, \cite{Huang:DAHAmodule}]
\label{thm:irr_O}
The $\H_q$-module $O(k_0,k_1,k_2,k_3)$ is irreducible if and only if 
$$
k_0^2,k_1^2,k_2^2,k_3^2\not=q^{-i}
\qquad 
\hbox{for $i=2,4,\ldots,d$}.
$$
\end{thm}

\begin{thm}
[Theorem 8.1, \cite{Huang:DAHAmodule}]
\label{thm:onto1_O}
If $V$ is a $(d+1)$-dimensional irreducible $\H_q$-module, then there exist nonzero scalars $k_0,k_1,k_2,k_3\in \F$ with $k_0 k_1 k_2 k_3=q^{-d-1}$ such that 
the $\H_q$-module $O(k_0,k_1,k_2,k_3)$ is isomorphic to $V$.
\end{thm}

Using Proposition \ref{prop:O}(i) a routine calculation yields the following lemma:

\begin{lem}
\label{lem:det_O}
The determinants of $t_0,t_1,t_2,t_3$ on $O(k_0,k_1,k_2,k_3)$ are equal to $k_0,k_1,k_2,k_3$, respectively.
\end{lem}

By means of Lemma \ref{lem:det_O} we develop the following discriminant to determine the nonzero scalars $k_0,k_1,k_2,k_3\in \F$ in Theorem \ref{thm:onto1_O}.

\begin{thm}
\label{thm:onto2_O}
Suppose that $V$ is a $(d+1)$-dimensional irreducible $\H_q$-module. For any nonzero scalars $k_0,k_1,k_2,k_3$ with $k_0k_1k_2k_3=q^{-d-1}$, the following are equivalent:
\begin{enumerate}
\item The $\H_q$-module $O(k_0,k_1,k_2,k_3)$ is isomorphic to $V$. 

\item The determinants of $t_0,t_1,t_2,t_3$ on $V$ are equal to $k_0,k_1,k_2,k_3$, respectively.
\end{enumerate}
\end{thm}
\begin{proof}
Immediate from Theorem \ref{thm:onto1_O} and Lemma \ref{lem:det_O}.
\end{proof}

\begin{lem}
\label{lem:t0t1_O}
The action of $t_0t_1$ on $O(k_0,k_1,k_2,k_3)$ is as follows:
\begin{align*}
(1-k_0 k_1 q^{2\lceil \frac{i}{2}\rceil} (t_0 t_1)^{(-1)^{i-1}})
v_i
=\left\{
\begin{array}{ll}
v_{i+1} 
\qquad 
&\hbox{for $i=0,1,\ldots,d-1$},
\\
0
\qquad 
&\hbox{for $i=d$}.
\end{array}
\right.
\end{align*}
\end{lem}
\begin{proof}
Apply Proposition \ref{prop:O} to evaluate the action of $t_0t_1$ on $O(k_0,k_1,k_2,k_3)$.
\end{proof}

\begin{lem}\label{lem:t0t2_O}
If the $\H_q$-module $O(k_0,k_1,k_2,k_3)$ is irreducible then the following hold:
\begin{enumerate}
\item There exists an $\F$-basis $\{w_i\}_{i=0}^d$ for $O(k_0,k_1,k_2,k_3)$ such that 
\begin{align*}
(1-k_0 k_1 q^{2\lceil \frac{i}{2}\rceil} (t_1 t_0)^{(-1)^{i-1}})
w_i
=\left\{
\begin{array}{ll}
w_{i+1} 
\qquad 
&\hbox{for $i=0,1,\ldots,d-1$},
\\
0
\qquad 
&\hbox{for $i=d$}.
\end{array}
\right.
\end{align*}

\item There exists an $\F$-basis $\{w_i\}_{i=0}^d$ for $O(k_0,k_1,k_2,k_3)$ such that 
\begin{align*}
(1-k_0 k_3 q^{2\lceil \frac{i}{2}\rceil} (t_3 t_0)^{(-1)^{i-1}})
w_i
=\left\{
\begin{array}{ll}
w_{i+1} 
\qquad 
&\hbox{for $i=0,1,\ldots,d-1$},
\\
0
\qquad 
&\hbox{for $i=d$}.
\end{array}
\right.
\end{align*}

\item There exists an $\F$-basis $\{w_i\}_{i=0}^d$ for $O(k_0,k_1,k_2,k_3)$ such that 
\begin{align*}
(1-k_0 k_2 q^{2\lceil \frac{i}{2}\rceil} (t_2 t_0)^{(-1)^{i-1}})
w_i
=\left\{
\begin{array}{ll}
w_{i+1} 
\qquad 
&\hbox{for $i=0,1,\ldots,d-1$},
\\
0
\qquad 
&\hbox{for $i=d$}.
\end{array}
\right.
\end{align*}
\end{enumerate}
\end{lem}
\begin{proof}
Suppose the $\H_q$-module $V=O(k_0,k_1,k_2,k_3)$ is irreducible.

(i): Since $t_0 t_1$ is similar $t_1 t_0$ on $V$, the part (i) is immediate from Lemma \ref{lem:t0t1_O}.

(ii): By Definition \ref{defn:H} there exists a unique $\F$-algebra automorphism $\rho:\H_q\to \H_q$ given by 
\begin{eqnarray*}
(t_0,t_1,t_2,t_3) 
&\mapsto &
(t_0,t_0^{-1} t_3 t_0,t_1 t_2 t_1^{-1}, t_1)
\end{eqnarray*}
whose inverse sends 
$(t_0,t_1,t_2,t_3)$ to
$(t_0,t_3,t_3^{-1} t_2 t_3,t_0 t_1 t_0^{-1})$. By Lemma \ref{lem:det_O} the determinants of $t_0,t_1,t_2,t_3$ on $V^\rho$ are $k_0,k_3,k_2,k_1$, respectively. Therefore the $\H_q$-module $V^\rho$ is isomorphic to 
$
O(k_0,k_3,k_2,k_1)
$
by Theorem \ref{thm:onto2_O}. 
It follows from Lemma \ref{lem:t0t1_O} that there exists an $\F$-basis $\{w_i\}_{i=0}^d$ for $V^\rho$ such that 
\begin{align*}
(1-k_0 k_3 q^{2\lceil \frac{i}{2}\rceil} (t_0 t_1)^{(-1)^{i-1}})
w_i
=\left\{
\begin{array}{ll}
w_{i+1} 
\qquad 
&\hbox{for $i=0,1,\ldots,d-1$},
\\
0
\qquad 
&\hbox{for $i=d$}.
\end{array}
\right.
\end{align*}
Observe that the action of $t_0t_1$ on $V^\rho$ is identical to the action of $t_3 t_0$ on $V$. Hence (ii) follows.

(iii): By Definition \ref{defn:H} there exists a unique $\F$-algebra automorphism $\rho:\H_q\to \H_q$ given by 
\begin{eqnarray*}
(t_0,t_1,t_2,t_3) 
&\mapsto &
(t_0,t_0^{-1} t_2t_0,t_0^{-1} t_3t_0,t_1)
\end{eqnarray*}
whose inverse sends 
$(t_0,t_1,t_2,t_3)$ to
$(t_0,t_3 ,t_0 t_1t_0^{-1},t_0 t_2t_0^{-1})$. By Lemma \ref{lem:det_O} the determinants of $t_0,t_1,t_2,t_3$ act on $V^\rho$ are $k_0,k_2,k_3,k_1$, 
respectively.
Therefore the $\H_q$-module $V^\rho$ is isomorphic to 
$
O(k_0,k_2,k_3,k_1) 
$
by Theorem \ref{thm:onto2_O}. 
It follows from Lemma \ref{lem:t0t1_O} that there exists an $\F$-basis $\{w_i\}_{i=0}^d$ for $V^\rho$ such that 
\begin{align*}
(1-k_0 k_2 q^{2\lceil \frac{i}{2}\rceil} (t_0 t_1)^{(-1)^{i-1}})
w_i
=\left\{
\begin{array}{ll}
w_{i+1} 
\qquad 
&\hbox{for $i=0,1,\ldots,d-1$},
\\
0
\qquad 
&\hbox{for $i=d$}.
\end{array}
\right.
\end{align*}
Observe that the action of $t_0t_1$ on $V^\rho$ is identical to the action of $t_2 t_0$ on $V$. Hence (iii) follows.
\end{proof}

Using Lemmas \ref{lem:t0t1_O} and \ref{lem:t0t2_O} yields the following equivalent conditions for 
$t_it_0$ and $t_0 t_i$
as diagonalizable on $O(k_0,k_1,k_2,k_3)$ for $i\in\{1,2,3\}$. Additionally we obtain the following sufficient conditions for $A,B,C$ as diagonalizable on $O(k_0,k_1,k_2,k_3)$.

\begin{lem}\label{lem:t1t0_diag_O}
If the $\H_q$-module $O(k_0,k_1,k_2,k_3)$ is irreducible the the following are equivalent:
\begin{enumerate}
\item $t_1 t_0$ and $t_0 t_1$ are diagonalizable on $O(k_0,k_1,k_2,k_3)$.

\item $t_1 t_0$ and $t_0 t_1$ are multiplicity-free on $O(k_0,k_1,k_2,k_3)$.

\item $k_0^2 k_1^2$ is not among $q^{-2},q^{-4},\ldots,q^{-2d}$.
\end{enumerate}
\end{lem}
\begin{proof}
Note that $t_1 t_0$ is similar to $t_0 t_1$ on $O(k_0,k_1,k_2,k_3)$.

(ii) $\Leftrightarrow$ (iii): 
By Lemma \ref{lem:t0t1_O} this characteristic polynomial of $t_0 t_1$ in $O(k_0,k_1,k_2,k_3)$ has the roots
\begin{gather}\label{eigen_t0t1_O}
k_0 k_1 q^i
\qquad \hbox{for $i=0,2,\ldots,d$},
\qquad 
k_0^{-1} k_1^{-1} q^{-i-1}
\qquad \hbox{for $i=1,3,\ldots,d-1$}
\end{gather} 
Since $q$ is not a root of unity, the scalars 
(\ref{eigen_t0t1_O}) are mutually distinct if and only if (iii) holds. Therefore (ii) and (iii) are equivalent.

(ii), (iii) $\Rightarrow$ (i): Trivial.

(i) $\Rightarrow$ (ii), (iii): By Lemma \ref{lem:t0t1_O} the product 
\begin{gather}\label{Yannihilator_O}
\prod_{i=0}^d (1-k_0 k_1 q^{2\lceil \frac{i}{2}\rceil} (t_0 t_1)^{(-1)^{i-1}})
\end{gather}
vanishes at $v_0$ and any proper factor of (\ref{Yannihilator_O}) does not vanish at $v_0$. Hence the $t_0 t_1$-annihilator of $v_0$ is equal to the characteristic polynomial of $t_0 t_1$ in $O(k_0,k_1,k_2,k_3)$. Therefore (i) implies (ii) and (iii).
\end{proof}

\begin{lem}\label{lem:t2t0_diag_O}
If the $\H_q$-module $O(k_0,k_1,k_2,k_3)$ is irreducible the the following are equivalent:
\begin{enumerate}
\item $t_2 t_0$ and $t_0 t_2$ are diagonalizable on $O(k_0,k_1,k_2,k_3)$.

\item $t_2 t_0$ and $t_0 t_2$ are multiplicity-free on $O(k_0,k_1,k_2,k_3)$.

\item $k_0^2 k_2^2$ is not among $q^{-2},q^{-4},\ldots,q^{-2d}$.
\end{enumerate}
\end{lem}
\begin{proof}
Applying Lemma \ref{lem:t0t2_O}(iii) the lemma follows by an argument similar to the proof of Lemma \ref{lem:t1t0_diag_O}.
\end{proof}

\begin{lem}\label{lem:t3t0_diag_O}
If the $\H_q$-module $O(k_0,k_1,k_2,k_3)$ is irreducible the the following are equivalent:
\begin{enumerate}
\item $t_3 t_0$ and $t_0 t_3$ are diagonalizable on $O(k_0,k_1,k_2,k_3)$.

\item $t_3 t_0$ and $t_0 t_3$ are multiplicity-free on $O(k_0,k_1,k_2,k_3)$.

\item $k_0^2 k_3^2$ is not among $q^{-2},q^{-4},\ldots,q^{-2d}$.
\end{enumerate}
\end{lem}
\begin{proof}
Applying Lemma \ref{lem:t0t2_O}(ii) the lemma follows by an argument similar to the proof of Lemma \ref{lem:t1t0_diag_O}.
\end{proof}

\begin{lem}\label{lem:ABC_diag_O}
If the $\H_q$-module $O(k_0,k_1,k_2,k_3)$ is irreducible then the following hold:
\begin{enumerate}
\item If $k_0^2 k_1^2$ is not among $q^{-2},q^{-4},\ldots,q^{2-2d}$ then $A$ is diagonalizable on $O(k_0,k_1,k_2,k_3)$.

\item If $k_0^2 k_3^2$ is not among $q^{-2},q^{-4},\ldots,q^{2-2d}$ then $B$ is diagonalizable on $O(k_0,k_1,k_2,k_3)$.

\item If $k_0^2 k_2^2$ is not among $q^{-2},q^{-4},\ldots,q^{2-2d}$ then $C$ is diagonalizable on $O(k_0,k_1,k_2,k_3)$.
\end{enumerate}
\end{lem}
\begin{proof}
(i): Let $\{w_i\}_{i=0}^d$ denote the $\F$-basis for $O(k_0,k_1,k_2,k_3)$ from Lemma \ref{lem:t0t2_O}(i). Given any element $X$ of $\H_q$, let $[X]$ denote the matrix representing $X$ with respect to $\{w_i\}_{i=0}^d$. 
Using Lemma \ref{lem:t0t2_O}(i) a direct calculation yields that $[A]$ is a lower triangular matrix of the form
\begin{gather}\label{[A]_O}
\begin{pmatrix}
\theta_0 & & & & & &{\bf 0}
\\
 &\theta_1
\\
  & &\theta_1
\\
  &  & &\theta_2
\\
& & & &\theta_2
\\
 &  & & & &\ddots
\\
&  & & & & &\theta_{\frac{d}{2}}
\\  
* & & & & & &0 &\theta_{\frac{d}{2}}
\end{pmatrix}
\end{gather}
where 
$$
\theta_i=k_0 k_1 q^{2i}+k_0^{-1} k_1^{-1} q^{-2i}
\qquad 
\hbox{for $i=0,1,\ldots,\frac{d}{2}$}.
$$

Since $q$ is not a root of unity and $k_0^2 k_1^2$ is not among $q^{-2},q^{-4},\ldots,q^{2-2d}$, the scalars $\{\theta_i\}_{i=0}^{\frac{d}{2}}$ are mutually distinct. Hence the $\theta_0$-eigenspace of $A$ in $O(k_0,k_1,k_2,k_3)$ has dimension one and the $\theta_i$-eigenspace of $A$ in $O(k_0,k_1,k_2,k_3)$ has dimension less than or equal to two for all $i=1,2,\ldots,\frac{d}{2}$. By (\ref{[A]_O}) the last two columns of $[A-\theta_{\frac{d}{2}}]$ are zero. By the rank-nullity theorem, the $\theta_\frac{d}{2}$-eigenspace of $A$ in $O(k_0,k_1,k_2,k_3)$ has dimension two.
To see the diagonalizability of $A$ it remains to show that the $\theta_i$-eigenspace of $A$ in $O(k_0,k_1,k_2,k_3)$ has dimension two for all $i=1,2,\ldots,\frac{d}{2}-1$.

By Lemma \ref{lem:t0t2_O}(i) the matrix $[t_1 t_0]$ is a lower triangular matrix of the form 
\begin{gather*}
\begin{pmatrix}
\vartheta_0 & & & & & &{\bf 0}
\\
 &\vartheta_1^{-1}
\\
  & &\vartheta_1
\\
  &  & &\vartheta_2^{-1}
  \\
&  &  & &\vartheta_2
\\
&  & & & &\ddots
\\
&  & & & & &\vartheta_{\frac{d}{2}}^{-1}
\\  
* & & & & & & &\vartheta_{\frac{d}{2}}
\end{pmatrix}
\end{gather*}
where 
$$
\vartheta_i=k_0 k_1 q^{2i}
\qquad 
\hbox{for $i=0,1,\ldots,\frac{d}{2}$}.
$$
Since $q$ is not a root of unity and $k_0^2 k_1^2\not\in\{q^{-2},q^{-4},\ldots,q^{2-2d}\}$ the eigenvalues $\{\vartheta_i^{\pm 1}\}_{i=1}^{\frac{d}{2}-1}$ of $t_1 t_0$ in $O(k_0,k_1,k_2,k_3)$ are multiplicity-free. Hence the $\theta_i$-eigenspace of $A$ in $O(k_0,k_1,k_2,k_3)$ has dimension two for all $i=1,2,\ldots,\frac{d}{2}-1$. Therefore (i) follows.

(ii): Applying Lemma \ref{lem:t0t2_O}(ii) the statement (ii) follows by an argument  similar to the part (i).

(iii): Applying Lemma \ref{lem:t0t2_O}(iii) the statement (iii) follows by an argument  similar to the part (i).
\end{proof}

Recall from the finite-dimensional irreducible $\triangle_q$-modules from \S\ref{s:AWmodule}. The composition factors of any $(d+1)$-dimensional irreducible $\H_q$-modules are classified as follows.

\begin{thm}
[\S 4.6, \cite{Huang:AW&DAHAmodule}]
\label{thm:CF_odd}
If the $\H_q$-module $O(k_0,k_1,k_2,k_3)$ is irreducible then the following hold:
\begin{enumerate}
\item If $d=0$ then the $\triangle_q$-module $O(k_0,k_1,k_2,k_3)$ is irreducible and it is isomorphic to 
\begin{align*}
V_0(k_0k_1,k_0 k_3, k_0 k_2).
\end{align*}

\item If $d\geq 2$ and $k_0^2=1$ then the factors of any composition series for the $\triangle_q$-module $O(k_0,k_1,k_2,k_3)$ are isomorphic to 
\begin{align*}
&V_0(k_0k_1,k_0 k_3, k_0 k_2),
\\
&V_{\frac{d}{2}-1}(k_0k_1q^{\frac{d}{2}+1},k_0 k_3 q^{\frac{d}{2}+1}, k_0 k_2 q^{\frac{d}{2}+1}),
\\
&V_{\frac{d}{2}-1}(k_0k_1q^{\frac{d}{2}+1},k_0 k_3 q^{\frac{d}{2}+1}, k_0 k_2 q^{\frac{d}{2}+1}).
\end{align*}

\item If $d\geq 2$ and $k_0^2\not=1$ then the factors of any composition series for the $\triangle_q$-module $O(k_0,k_1,k_2,k_3)$ are isomorphic to
\begin{align*}
&V_{\frac{d}{2}}(k_0k_1q^{\frac{d}{2}},k_0 k_3 q^{\frac{d}{2}}, k_0 k_2 q^{\frac{d}{2}}),
\\
&V_{\frac{d}{2}-1}(k_0k_1q^{\frac{d}{2}+1},k_0 k_3 q^{\frac{d}{2}+1}, k_0 k_2 q^{\frac{d}{2}+1}).
\end{align*}
\end{enumerate}
\end{thm}

Applying Lemma \ref{lem:dia_UAW} to Theorem \ref{thm:CF_odd} yields the following lemmas:

\begin{lem}\label{lem:diag_ABC_O_k02=1}
Suppose that the $\H_q$-module $O(k_0,k_1,k_2,k_3)$ is irreducible. If $k_0^2=1$ then the following hold:
\begin{enumerate}
\item $A$ 
is multiplicity-free on all composition factors of the $\triangle_q$-module $O(k_0,k_1,k_2,k_3)$ if and only if 
$k_1^2$ is not among $q^{-6},q^{-8},\ldots,q^{2-2d}$.

\item $B$
is multiplicity-free on all composition factors of the $\triangle_q$-module $O(k_0,k_1,k_2,k_3)$ if and only if 
$k_3^2$ is not among $q^{-6},q^{-8},\ldots,q^{2-2d}$.

\item $C$ 
is multiplicity-free on all composition factors of the $\triangle_q$-module $O(k_0,k_1,k_2,k_3)$ if and only if 
$k_2^2$ is not among $q^{-6},q^{-8},\ldots,q^{2-2d}$.
\end{enumerate}
\end{lem}

\begin{lem}\label{lem:diag_ABC_O_k02neq1}
Suppose that the $\H_q$-module $O(k_0,k_1,k_2,k_3)$ is irreducible. If $k_0^2\not=1$ then the following hold:
\begin{enumerate}
\item $A$ 
is multiplicity-free on all composition factors of the $\triangle_q$-module $O(k_0,k_1,k_2,k_3)$ if and only if 
$k_0^2 k_1^2$ is not among $q^{-2},q^{-4},\ldots,q^{2-2d}$.

\item $B$
is multiplicity-free on all composition factors of the $\triangle_q$-module $O(k_0,k_1,k_2,k_3)$ if and only if 
$k_0^2 k_3^2$ is not among $q^{-2},q^{-4},\ldots,q^{2-2d}$.

\item $C$ 
is multiplicity-free on all composition factors of the $\triangle_q$-module $O(k_0,k_1,k_2,k_3)$ if and only if 
$k_0^2 k_2^2$ is not among $q^{-2},q^{-4},\ldots,q^{2-2d}$.
\end{enumerate}
\end{lem}

We are in the position to prove Theorem \ref{thm:diagonal} in odd-dimensional case:

\begin{thm}\label{thm:diagonal_odd}
If $V$ is an odd-dimensional irreducible $\H_q$-module then the following are equivalent:
\begin{enumerate}
\item $A$ {\rm (}resp. $B${\rm )} {\rm (}resp. $C${\rm )} is diagonalizable on $V$.

\item $A$ {\rm (}resp. $B${\rm )} {\rm (}resp. $C${\rm )} is diagonalizable on all composition factors of the $\triangle_q$-module $V$.

\item $A$ {\rm (}resp. $B${\rm )} {\rm (}resp. $C${\rm )} is multiplicity-free on all composition factors of the $\triangle_q$-module $V$.
\end{enumerate}
\end{thm}
\begin{proof}
(i) $\Rightarrow$ (ii): Trivial.

(ii) $\Rightarrow$ (iii): Immediate from Lemma \ref{lem:dia_UAW}.

(iii) $\Rightarrow$ (i): 
Suppose that (iii) holds. 
Let $d=\dim V-1$. 
By Theorem \ref{thm:onto_E} there are nonzero scalars $k_0,k_1,k_2,k_3\in \F$ with $k_0 k_1 k_2 k_3=q^{-d-1}$ such that the $\H_q$-module $O(k_0,k_1,k_2,k_3)$ is isomorphic to $V$.

Suppose that $k_0^2=1$. Since the $\H_q$-module $O(k_0,k_1,k_2,k_3)$ is irreducible  it follows from Theorem \ref{thm:irr_O} that none of $k_1^2,k_2^2,k_3^2$ is among $q^{-2},q^{-4},\ldots,q^{-d}$. Combined with Lemmas \ref{lem:ABC_diag_O} and \ref{lem:diag_ABC_O_k02=1} this yields that (i) holds for the case $k_0^2=1$. By Lemmas \ref{lem:ABC_diag_O} and \ref{lem:diag_ABC_O_k02neq1} the statement (i) holds for the case $k_0^2\not=1$. Therefore (i) follows.
\end{proof}

We end this paper with an example to show that the condition \ref{XY} is not a necessary condition for $A$ and $B$ as diagonalizable on odd-dimensional irreducible $\H_q$-modules.

\begin{exam}\label{exam:b_O}
Choose $k_0$ as a nonzero scalar in $\F$ satisfying 
$$
k_0^2\not\in\{q^{-2},q^{-4},\ldots,q^{2-3d}\}
$$
and set 
$
(k_1,k_2,k_3)=
(q^{-d} k_0^{-1},
q^{d-1} k_0,
q^{-d} k_0^{-1})$. 
Note that $k_0 k_1 k_2 k_3=q^{-d-1}$. 
By Theorem \ref{thm:irr_O} the $\H_q$-module $O(k_0,k_1,k_2,k_3)$ is irreducible.
It follows from Lemmas \ref{lem:t1t0_diag_O} and \ref{lem:t3t0_diag_O} that neither of $t_0 t_1$ and $t_3 t_0$ is diagonalizable on $O(k_0,k_1,k_2,k_3)$. However the elements $A$ and $B$ are diagonalizable on $O(k_0,k_1,k_2,k_3)$ by Lemma \ref{lem:ABC_diag_O}. 
\end{exam}

\subsection*{Acknowledgements}

The research is supported by the Ministry of Science and Technology of Taiwan under the project MOST 106-2628-M-008-001-MY4.

\bibliographystyle{amsplain}
\bibliography{MP}

\end{document}